\documentclass[english,a4paper,nolineno]{socg-lipics-v2021}

\usepackage[utf8]{inputenc}
\usepackage{scalerel}
\usepackage{enumerate}
\usepackage{stmaryrd}
\usepackage{tikz-cd}
\usepackage{amsmath,amscd,amssymb}
\usepackage{amsthm}
\usepackage{amsfonts}
\usepackage{mathtools}
\usepackage{color}
\usepackage[mathscr]{eucal}
\usepackage{amssymb}
\usepackage{comment}
\usepackage{algorithm}
\usepackage[]{algpseudocode}
\usepackage{chngcntr}
\usepackage{apptools}
\usepackage{hyperref}

\hideLIPIcs

\newcommand{\real}{\mathbb{R}}

\newcommand{\dgh}{d_{\mathrm{GH}}}

\newcommand{\db}{d_\mathrm{b}}

\newcommand{\diam}{\mathrm{diam}}

\newcommand{\Barc}{\mathcal{BC}}

\newcommand{\vr}{\mathrm{VR}}

\newcommand{\F}[1][]{\mathcal{F}_{#1}}
\newcommand{\MF}[1][]{\mathcal{MF}_{#1}}

\newcommand{\reebradius}{\rho}
\newcommand{\reebgraph}{R}

\newcommand{\corrR}{\EuScript{R}}
\newcommand{\corrS}{\EuScript{S}}

\newtheorem*{theorem*}{Theorem}

\graphicspath{{figures/}}

\newcommand{\define}[1]{\textbf{#1}}

\EventEditors{Wolfgang Mulzer and Jeff M. Phillips}
\EventNoEds{2}
\EventLongTitle{40th International Symposium on Computational Geometry (SoCG 2024)}
\EventShortTitle{SoCG 2024}
\EventAcronym{SoCG}
\EventYear{2024}
\EventDate{June 11-14, 2024}
\EventLocation{Athens, Greece}
\EventLogo{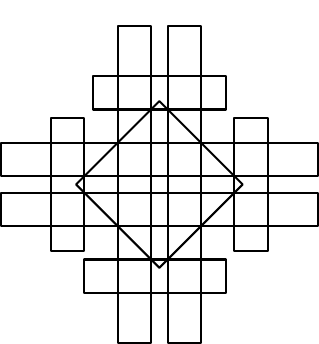}
\SeriesVolume{293}
\ArticleNo{36}

\title{Stability and Approximations for Decorated Reeb Spaces}

\authorrunning{Justin Curry, Washington Mio, Tom Needham, Osman Okutan, Florian Russold}

\author{Justin Curry}{University at Albany, State University of New York, USA }{jmcurry@albany.edu}{https://orcid.org/
0000-0003-2504-8388}{Supported by NSF CCF-1850052 and NASA 80GRC020C0016}

\author{Washington Mio}{Florida State University, Tallahassee, Florida}{wmio@fsu.edu}{
}{Supported by NSF DMS-1722995}

\author{Tom Needham}{Florida State University, Tallahassee, Florida}{tneedham@fsu.edu}{https://orcid.org/0000-0001-6165-3433}{Supported by NSF DMS-2107808 and DMS-2324962}

\author{Osman Berat Okutan}{Max Planck Institute for Mathematics in the Sciences, Leipzig, Germany}{osman.okutan@mis.mpg.de}{
}{}

\author{Florian Russold}{Graz University of Technology, Graz, Austria}{russold@tugraz.at}{}{Supported by the Austrian Science Fund (FWF): W1230}

\Copyright{Justin Curry, Washington Mio, Tom Needham, Osman Okutan, Florian Russold} 



\funding{} 

\relatedversiondetails[cite={curry2023stability}]{Full Version}{https://arxiv.org/abs/2312.01982}

\begin{document}

\maketitle

\begin{abstract}
 Given a map $f:X \to M$ from a topological space $X$ to a metric space $M$, a \emph{decorated Reeb space} consists of the Reeb space, together with an attribution function whose values recover geometric information lost during the construction of the Reeb space. 
 For example, when $M=\mathbb{R}$ is the real line, the Reeb space is the well-known Reeb graph, and the attributions may consist of persistence diagrams summarizing the level set topology of $f$. 
 In this paper, we introduce decorated Reeb spaces in various flavors and prove that our constructions are Gromov-Hausdorff stable. We also provide results on approximating decorated Reeb spaces from finite samples and leverage these to develop a computational framework for applying these constructions to point cloud data.
\end{abstract}

\section{Introduction}

Graphical summaries of topological spaces have a long history in pure mathematics \cite{kronrod1950functions,reeb1946points} and have more recently served as popular tools for shape and data analysis; prominent examples include Reeb graphs~\cite{ge2011data}, Mapper graphs~\cite{singh2007topological} and merge trees~\cite{morozov2013interleaving}. The related concept of persistent homology is another ubiquitous tool for topological summarization in modern data science~\cite{carlsson2014topological}. Graphical and persistence-based summaries frequently capture complementary information about a dataset, and there is a recently developing body of work which studies methods for enriching graphical descriptors with additional geometric or topological  data~\cite{curry2022decorated,curry2023convergence,curry2023topologically,li2023comparing,li2023flexible}.

In this paper we prove a Gromov-Hausdorff stability result for an enriched topological summary called the \define{Decorated Reeb Space}. The Reeb graph summarizes connectivity of level sets of a  function $f:X \to \real$ defined on some topological space $X$; the Reeb space generalizes the Reeb graph to summarize maps $f:X\to M$ valued in some metric space $M$---such considerations are natural when dealing with multi-variate data~\cite{edelsbrunner2008reeb}.
The enrichment presented in this paper synthesizes persistent homology with the Reeb space to provide a single topological descriptor more powerful than persistence or Reeb spaces alone.
Before summarizing the content of the paper, we first provide an illustration of the decorated Reeb space pipeline in order to shed light on the nature and importance of our results.

\begin{figure}
    \centering
    \includegraphics[width = \textwidth]{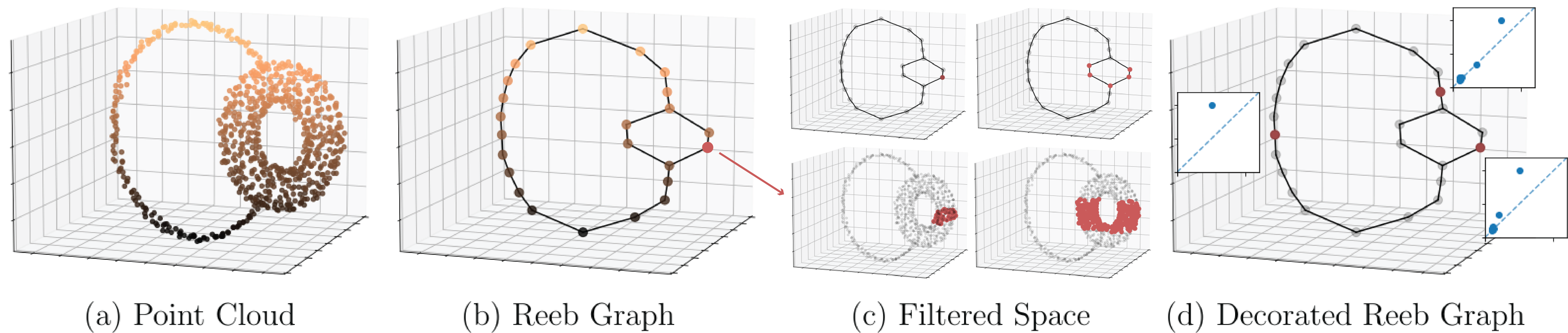}
    \caption{The Decorated Reeb Graph pipeline. (a) A {\bf point cloud}, noisily sampled from a space homotopy equivalent to $(S^1 \times S^1) \vee S^1$. (b) The (estimated) {\bf Reeb graph} associated to the point cloud, with respect to the height function. (c) To each node of the Reeb graph, we associate a {\bf filtered space}, which is a sequence of subsets of the original point cloud. (d) Taking persistent homology of the filtered space associated to each node yields a {\bf decorated Reeb graph}; each node is attributed with a persistence diagram.}
    \label{fig:pipeline}
\end{figure}

\subparagraph*{Illustration of the Pipeline.} Consider Figure \ref{fig:pipeline}. 
Part (a) shows a point cloud in $\real^3$, sampled (with additive noise) from a space that is homotopy equivalent to $(S^1 \times S^1)\vee S^1$. 
Using height along the $z$-axis as a function on the point cloud, we obtain an estimation of the Reeb graph of the space, as shown in (b). 
To each node of the Reeb graph, we assign a filtration of the original point cloud, as illustrated in (c). 
For a fixed anchor node in the Reeb graph (highlighted in red) and for each scale $r$ we include those points in the original point cloud that correspond to nodes in the Reeb graph within Reeb radius $r$ (see Definition \ref{defn:reeb_radius}) of the anchor node. 
The top row of (c) shows subsets of the Reeb graph within two different Reeb radii and the bottom row shows the original point clouds within those radii of the anchor node. 
To this filtration we obtain a bi-filtered simplicial complex by considering the Vietoris-Rips complex (with proximity parameter $s$) at each Reeb radius $r$. 
Taking a one-dimensional slice of the $(r,s)$-parameter space, we compute the (one-dimensional) persistent homology and add this as an attribution to the anchor node in the Reeb graph. 
By letting the anchor node vary across the Reeb graph, and repeating this process, we obtain a novel structure in TDA called the \define{Decorated Reeb Graph}. 
Part (d) of Figure \ref{fig:pipeline} shows the resulting persistence diagram for three different anchor nodes. 
We note that the persistent homology decoration at each node captures the local topology of the original point cloud, and this construction is more sensitive to topological features that exist near each anchor node. 
In our example, the attributions on the right side of the Reeb graph capture the topology of the torus (there are two off-diagonal points in the $H_1$ persistence diagram), whereas the attribution on the left hand side captures only the cycle coming from the larger circle.

\subparagraph*{Summary of Main Results.} The output signature of the Decorated Reeb Graph pipeline described above---that is, a graph whose nodes are attributed with persistence diagrams---is the same as that of the construction considered in \cite{curry2023topologically}. However, as noted in \cite[Remark 5.1]{curry2023topologically}, the computational pipeline in that paper is largely divorced from the theory, and the abstract stability results there have little relevance to the algorithmic implementation. 
This paper gives a novel theoretical construction of decorated Reeb spaces arising from functions $f:X \to M$ defined on connected spaces, borrowing tools from metric geometry. Our main contributions are:
\begin{itemize}
    \item We introduce the concept of the \define{Reeb radius} associated to a map $f:X \to M$ and show that this is closely tied to Reeb graph smoothings considered in~\cite{de2016categorified}; see Theorem \ref{thm:reeb-radius-smoothings}.
    \item We show that, under (quantifiable) tameness assumptions on the spaces and functions, if $f:X \to M$ and $g:Y \to M$ are close in a certain Gromov-Hausdorff sense, then so are their decorated Reeb spaces; see Theorem \ref{thm:barcode-decoration-stability}.
    \item We show that a tame function $f:X \to M$ can be approximated by a finite graph so that the (continuous) decorated Reeb space of $X$ is well approximated by the (discrete) decorated Reeb space of the graph; see Theorem \ref{thm:finite-approximation}. 
    \item Finally, we illustrate our constructions via computational examples in Section \ref{subsec:computation}.
\end{itemize}

\section{Decorated Reeb Spaces}\label{sec:decorated_reeb_spaces}

This section introduces the main constructions that we study throughout the paper. 
We begin with some preliminary definitions, which allows us to set conventions and notation.

\subsection{Reeb Spaces and Smoothings}

Let $X$ be a topological space, $(M,d_M)$ be a metric space, and $f:X \to M$ be a map. 
In this paper, every topological space is assumed to be compact, connected, and locally path connected, and every map is assumed to be continuous. We refer to the data $(X,f)$ as an \define{$M$-field}, and denote the class of all $M$-fields as $\F[M]$.

The \define{Reeb space} associated to $(X,f) \in \F[M]$ is defined to be the quotient space of $X$ under the equivalence relation $\sim$ on $X$ defined by $x \sim x'$ if and only if $x$ and $x'$ lie in the same connected component of a level set of $f$. 
This construction was independently conceived by Georges Reeb \cite{reeb1946points} and Aleksandr Kronrod \cite{kronrod1950functions}, but the attribution to Kronrod is frequently dropped. 
Keeping with this convention, we denote the Reeb space by $\reebgraph_f = X/\sim$ and the quotient map by $\pi_f: X \to \reebgraph_f$. 
For $x \in X$, we denote the equivalence class $\pi_f(x)$ by $[x]$. 
By definition, the map $f:X\to M$ factors through $\reebgraph_f$ to define an induced map $\bar{f}: \reebgraph_f \to M$, i.e., $\bar{f}([x])=f(x)$. 
When $M = \real$, the resulting quotient $\reebgraph_f$ is typically called the \define{Reeb graph}, although additional hypotheses are required to ensure a true graph structure \cite{sharko2006kronrod}. 
These assumptions are automatically satisfied in most computational settings and one should consult \cite{ge2011data,wang2021point} for a sample of the applications to data analysis.

Both the Reeb graph and Reeb space are sensitive to perturbations of the map $f:X\to M$.
One way of remedying this is to consider analogs of the smoothing operation first introduced in \cite{de2016categorified}.
Given $\epsilon>0$, the \define{$\epsilon$-smoothing} of the Reeb space $\reebgraph_f$,  denoted $R^{\epsilon}_f$, is defined in a two-step process.
First one constructs the space $T^\epsilon_f$ and map $f^\epsilon:T^\epsilon \to M$ defined by
\begin{equation}\label{eqn:f-epsilon}
 T^{\epsilon}_f:=\{(x,v) \in X \times M: d_M(f(x),v) \leq \epsilon\}, \qquad f^{\epsilon}(x,v) := v.   
\end{equation}
If $\reebgraph_{f^\epsilon}$ is the Reeb space of  $(T^\epsilon,f^\epsilon)$, then $R^{\epsilon}_f$ is the image of $X$ under the composition,
\begin{equation}\label{eqn:pi-epsilon}
 \pi_f^\epsilon : X\hookrightarrow T^{\epsilon}_f \to \reebgraph_{f^{\epsilon}} \quad \text{where} \quad x \mapsto (x,f(x)) \mapsto [(x,f(x))],   
\end{equation}
of the inclusion and quotient map.
We note that $\reebgraph_f^0$ can be identified with $\reebgraph_f$, so this smoothing operation produces a 1-parameter family of Reeb spaces starting with $\reebgraph_f$.

\subsection{Reeb Radius}

A key concept introduced in this paper is the  \emph{Reeb radius}, which contains all information about which points in $X$ are identified in the Reeb space and its associated smoothings. 
Later, we will use this to metrize our Reeb space and to induce decorations.

\begin{figure}
    \centering
    \includegraphics[width = .6\textwidth]{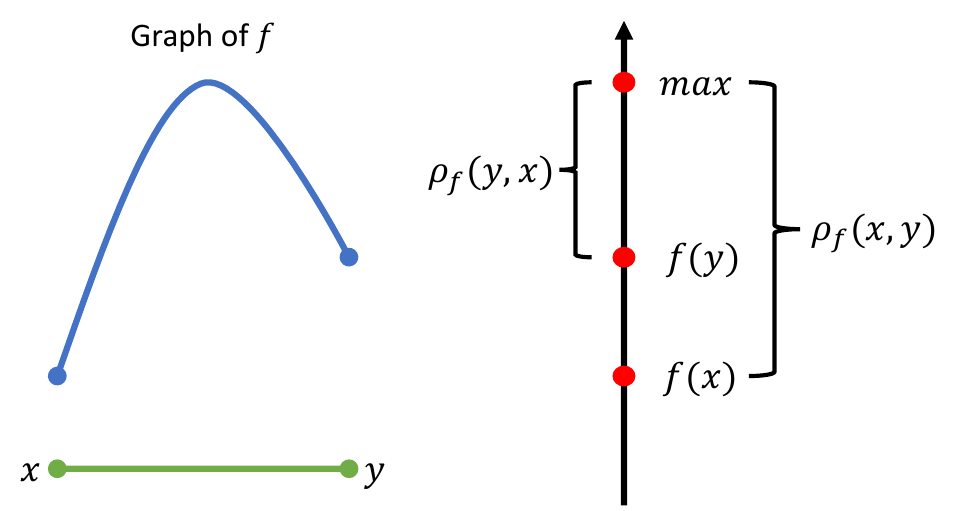}
    \caption{The graph of a generic function $f:[0,1]\to\mathbb{R}$ is its own Reeb graph. Notice that the Reeb radius is not symmetric in $x$ and $y$.}
    \label{fig:reeb-radius}
\end{figure}

\begin{definition}[Reeb radius]\label{defn:reeb_radius}
   For $(X,f) \in \F[M]$, define the \define{Reeb radius} $\reebradius_f: X \times X \to \real$ by
    \begin{equation}
        \reebradius_f(x,y) := \inf_\gamma \sup_t d_M(f(x),f(\gamma(t))),
    \end{equation}
    where the infimum is taken over all paths $\gamma:[0,1] \to X$ with $\gamma(0)=x$, $\gamma(1)=y$. 
\end{definition}

The Reeb radius is very similar to the \define{Reeb distance} of \cite{bauer2014measuring}. 
This is the map
\begin{equation}\label{eqn:reeb_distance}
\partial_f:X \times X \to \real, \qquad \partial_f(x,y) := \inf_\gamma \diam\big(f(\gamma([0,1]))\big),
\end{equation}
where the infimum is also taken over continuous paths from $x$ to $y$. 
It is straightforward to show that $\reebradius_f \leq \partial_f \leq 2\reebradius_f$, but one clear difference is that the Reeb radius is not necessarily symmetric; see Figure \ref{fig:reeb-radius} for an example.

We now show how the Reeb radius characterizes Reeb smoothings.

\begin{theorem}\label{thm:reeb-radius-smoothings}
    For $(X,f) \in \F[M]$ and $x, y \in X$, $\pi_f^\epsilon(x)=\pi_f^\epsilon(y)$ if and only if $f(x)=f(y)$ and $\reebradius_f(x,y) \leq \epsilon$.
\end{theorem}

\begin{lemma}\label{lem:reeb_radius_component}
    For $x,y \in X$, $y$ is contained in the connected component of $x$ in $f^{-1}(\mathbb{D}^r(f(x)))$ if and only if $\reebradius_f(x,y) \leq r$, where $\mathbb{D}^r(v)$ denotes the closed ball with radius $r$ and center $v$.
\end{lemma}
\begin{proof}
    Assume $\reebradius_f(x,y) \leq r$. 
    For each integer $n>0$, let $C_n$ denote the connected component of $x$ in $f^{-1}(\mathbb{D}^{r+1/n}(f(x)))$. 
    Note that $y \in C_n$ for all $n$. 
    The collection $\{C_n\}$ forms a decreasing family of compact connected sets. 
    By \cite[Corollary~6.1.19]{engelking1989general}, $C:=\cap C_n$ is connected and $x,y \in C \subseteq f^{-1}(\mathbb{D}^r(f(x)))$.

    Now, assume that $y$ is contained in the connected component of $x$ in $f^{-1}(\mathbb{D}^r(f(x)))$. For an integer $n>0$, let $U_n$ denote the connected component of $x$ in $f^{-1}(\mathbb{B}^{r+1/n}(f(x)))$, where $\mathbb{B}^{r+1/n}$ denotes the \emph{open} ball with radius $r+1/n$. 
    We have $y \in U_n$, which is path connected, since it is open, connected, and $X$ is locally path connected. 
    Thus, there exists a path in $U_n$ from $x$ to $y$, so that $\reebradius_f(x,y) \leq r + 1/n$. 
    Since $n>0$ was arbitrary, we have $\reebradius_f(x,y) \leq r$.
\end{proof}

\begin{proof}[Proof of Theorem \ref{thm:reeb-radius-smoothings}]
    Let $x \in X$ and $f(x)=v$. 
    By Equation \eqref{eqn:f-epsilon} we have that
    \[
    (f^\epsilon)^{-1}(v)=\{(x',v):d_M(f(x),v) \leq \epsilon \}=f^{-1}(\mathbb{D}^\epsilon(v)) \times \{v\},
    \] 
    where we continue to use $\mathbb{D}^\epsilon$ for a closed $\epsilon$-ball in $M$. Hence, by Equation \eqref{eqn:pi-epsilon}, $\pi_f^\epsilon(x)=\pi_f^\epsilon(y)$ if and only if $f(y)=f(x)$ and $y$ is in the connected component of $x$ in $f^{-1}(\mathbb{D}^\epsilon(f(x)))$. This happens if and only if $\reebradius_f(x,y) \leq \epsilon$ by Lemma \ref{lem:reeb_radius_component}.
\end{proof}

\subsection{Metrizing the Reeb Space}

One of our main goals is to establish Gromov-Hausdorff stability of our constructions. In order to formulate these results, we first need to metrize the Reeb space $\reebgraph_f$. 
We do this using the Reeb radius, but first we need to establish some of its properties. 

\begin{proposition}\label{prop:reeb_radius_properties}
    Let $(X,f) \in \F[M]$. For all $x,y,z \in X$, we have:
    \begin{enumerate}[i)]
        \item $\reebradius_f(x,z) \leq \reebradius_f(x,y)+\reebradius_f(y,z)$.
        \item $\reebradius_f(x,y)=0$ if and only if $[x]=[y]$.
        \item If $[x]=[x']$ and $[y]=[y']$, then $\reebradius_f(x,y)=\reebradius_f(x',y')$.
    \end{enumerate}
\end{proposition}

\begin{remark}
    With these properties, $\reebradius_f$ induces a well defined \emph{quasimetric} on $\reebgraph_f$, meaning that it only lacks the symmetry axiom of a metric; recall Figure \ref{fig:reeb-radius}.
\end{remark}

\begin{proof}
\begin{enumerate}[i)]
    \item Given paths $\alpha$ from $x$ to $y$ and $\beta$ from $y$ to $z$, let $\gamma$ denote the path from $x$ to $z$ obtained by concatenating $\alpha$ and $\beta$. We have
    \begin{equation*}
        \begin{split}
            \reebradius_f(x,z) &\leq \sup_t d_M(f(x),f(\gamma(t))) = \max\left(\sup_t d_M(f(x),f(\alpha(t))),\sup_t d_M(f(x),f(\beta(t)))\right) \\
            &\leq \max\left(\sup_t d_M(f(x),f(\alpha(t))),d_M(f(x)+f(y))+\sup_t d_M(f(y),f(\beta(t)))\right) \\
            &\leq \sup_t d_M(f(x),f(\alpha(t)))+\sup_t d_M(f(y),f(\beta(t))).
        \end{split}
    \end{equation*}
    Since this holds for arbitrary paths $\alpha$ and $\beta$, the triangle inequality follows.
    \item This is the $\epsilon = 0$ case of Theorem \ref{thm:reeb-radius-smoothings}.

    \item Using part {\bf i)}, $\reebradius_f(x,y) \leq \reebradius_f(x,x')+\reebradius_f(x',y')+\reebradius_f(y',y)=\reebradius_f(x',y')$. Repeating this argument with the roles of $x,y$ and $x',y'$ reversed proves $\reebradius_f(x',y') \leq \reebradius_f(x,y)$.
\end{enumerate}
\end{proof} 

We now obtain a metric on $\reebgraph_f$ by symmetrizing the Reeb radius.

\begin{definition}\label{def:reeb_distance}
    For $(X,f) \in \F[M]$, we define $d_f: \reebgraph_f \times \reebgraph_f \to [0,\infty)$ by 
    \[
    d_f([x],[y]):=2\max(\reebradius_f(x,y),\reebradius_f(y,x)).
    \]
\end{definition}

As already remarked, it is not hard to show that $d_f$ is bi-Lipschitz equivalent to the metric induced by the Reeb distance (Equation \eqref{eqn:reeb_distance}), but they are not equal in general.

\subsection{Decorations}\label{sec:decorations}

We wish to endow the Reeb space $\reebgraph_f$ of a map $f:X \to M$ with an additional function $\reebgraph_f \to Y$, for some (perhaps extended- and/or pseudo-) metric space $Y$; this function will be our attribution or \define{decoration} function. 
The metric space $Y$ will typically be the space of all (multi-parameter) persistence modules or barcodes; the decoration will capture topological information about $X$ that is lost in the quotient $\reebgraph_f$. 

In more detail, recall that Topological Data Analysis (TDA) typically views a function $f:X \to \real$ as a stepping stone to defining a \define{filtration}---an increasing chain of closed subsets that exhausts $X$.
As such, an $\real$-field $(X,f)$ can be equivalently regarded as a \define{filtered space}---a space along with its exhaustion $\{f^{-1}(-\infty,t]\}_{t\in \real}$.
In this sense, our first decoration function assigns a filtration to each point of the Reeb space. 

\begin{definition}[Filtration decoration for Reeb Spaces]\label{def:filter_decoration}
    Let $(X,f) \in \F[M]$. The \define{filtration decoration} for $\reebgraph_f$ is the map $D_f:\reebgraph_f \to \F[\real]$ defined by
    \[
    D_f([x]) := (X,\reebradius_f(x,\cdot)),
    \]
    where $\reebradius_f$ is the Reeb radius; see Figure \ref{fig:pipeline}(c) for an example.
\end{definition}

Filtration or filtered space decorations are somewhat unwieldy to work with computationally, but these structures can be summarized using persistent homology. To this end, we assume the reader is familiar with the basic notions of persistent homology: the \define{Vietoris-Rips complex}, the \define{persistent homology barcode}, and the \define{bottleneck distance} between barcodes; see \cite{carlsson2014topological} for a reminder. In the following, we consider a \define{simplicial filtration}; in general, this is a family $\{K^r\}_{r \in P}$ of simplicial complexes indexed by a poset $(P,\leq)$ such that $K^r$ is a subcomplex of $K^s$ whenever $r \leq s$. In this paper, we consider \define{1-parameter} and \define{2-parameter simplicial filtrations}, where $P = \real_{\geq 0}$ and $P=\real^2_{\geq 0}$ (with its product poset structure), respectively. In particular, we use $\{\mathrm{VR}^r(Y)\}_{r \geq 0}$ for the Vietoris-Rips complex of a metric space $Y$. We now introduce our main computational object of study, the barcode decoration, which is an attribution valued in $\Barc$---the space of barcodes with the bottleneck distance.

\begin{definition}[Barcode decoration for Reeb Spaces]\label{def:barcode_decoration}
    Let $(X,f) \in \F[M]$ and further assume that $X$ is a metric space. 
    Let $\lambda, c \geq 0$ and let $k$ be a non-negative integer. 
    The \define{barcode decoration} for $\reebgraph_f,\lambda, c, k$ is the map $B_{f,k}^{\lambda, c}:\reebgraph_f \to \Barc$, where $B^{\lambda, c}_{f,k}([x])$ is the $k$-dimensional persistent homology barcode of the Vietoris-Rips $1$-parameter simplicial filtration
    \[
    \big\{\vr^r(\{y \in X: \reebradius_f(x,y) \leq \lambda r + c \}) \big\}_{r \geq 0}.
    \]
\end{definition}

See Figures \ref{fig:pipeline} and \ref{fig:camel_examples} for examples.

\begin{remark}
    In fact, we should only consider $B_{f,k}^{\lambda, c}([x])$ to be a \emph{persistence module}, as the tameness conditions for the guaranteed existence of a barcode decomposition~\cite{crawley2015decomposition} may not hold. With a view toward computation, we will abuse terminology and still refer to the decorations as barcodes, but our results hold when decorations are considered instead as persistence modules and the interleaving distance is used in place of bottleneck distance.
\end{remark}

\section{Gromov-Hausdorff Stability}

In this section, we establish stability of the decorated Reeb space constructions defined above.

\subsection{Metric Fields and Multiscale Comparisons}

As Definition \ref{def:barcode_decoration} intimates, filtration-attributed Reeb spaces can be compared using the Vietoris-Rips filtration if and only if $X$ has a metric.
This will be essential as we move from a connected topological space $X$ to a point-sampling of $X$ along with its map $f:X\to M$.
To emphasize this extra structure, we call the pair $(X,f)$ an \define{$M$-metric field} when $X$ has a metric (otherwise, $f:X\to M$ is just an $M$-field).
Recall that the collection of all $M$-fields is denoted $\F[M]$. By contrast, we will denote the collection of all $M$-metric fields by $\MF[M]$. Methods for metrizing $\MF[M]$ can be found in \cite{anbouhi2023metrics}, but here we will focus on developing variations on the Gromov-Hausdorff distance to compare metric fields.

Recall that a \define{correspondence} between sets $X$ and $Y$ is a subset $\corrR \subset X \times Y$ such that the coordinate projections are surjective when restricted to $\corrR$.

\begin{definition}[(r,s)-Correspondences between metric fields]
    Let $(X,f), (Y,g) \in \MF[M]$, let $\corrR$ be a correspondence between $X,Y$, and let $r,s \geq 0$. 
    We call $\corrR$ an \define{$(r,s)$-correspondence} between $(X,f)$ and $(Y,g)$ if for all $(x,y)$, $(x',y') \in \corrR$ we have
    \[
    |d_X(x,x')-d_Y(y,y')| \leq 2r \qquad \text{and} \qquad d_M(f(x),g(y)) \leq s. 
    \]
\end{definition}

We use the following metrization of $\MF[M]$ in our stability results below.

\begin{definition}\label{def:field_GH}
    The \define{Gromov-Hausdorff distance} between $M$-metric fields $(X,f)$ and $(Y,g)$, denoted $\dgh((X,f),(Y,g))$, is defined to be the infimum over $r$ such that there is an $(r,r)$-correspondence between $(X,f)$ and $(Y,g)$.
\end{definition}

When $M$ is a single point, the above definition reduces to the regular Gromov-Hausdorff distance between metric spaces. The case $M = \real$ reduces to the definition given in \cite[Definition~2.4]{chazal2009gromov}. 
This definition gives a pseudometric on metric fields.

The filtration decoration of Definition \ref{def:filter_decoration} produces a metric field valued in metric fields. That is, for $(X,f) \in \MF[M]$ and $[x] \in \reebgraph_f$, the filtration decoration $D_f([x]) = (X,\reebradius_f(x,\cdot))$ is naturally an $\real$-metric field, i.e., $D_f$ is a map $D_f:\reebgraph_f \to \MF[\real]$. The class $\MF[\real]$ can be endowed with the GH pseudometric above, so that $(X,D_f) \in \mathcal{MF}_{\mathcal{MF}_{\real}}$ (where we consider $D_f$ as a function on $X$ by composing with the Reeb quotient map).
The following definition shows how we can compare these decorations.

\begin{definition}[(r,s,t)-Correspondences between metric field-valued functions]
    Let $(M,d_M)$ be a metric space. 
    Let $(X,d_X)$ and $(Y,d_Y)$ be metric spaces endowed with functions $D_X: X \to \MF[M]$ and $D_Y: Y \to \MF[M]$---that is, $(X,D_X),(Y,D_Y) \in \mathcal{MF}_{\mathcal{MF}_M}$. 
    Let $\corrR$ be a correspondence between $X$ and $Y$. 
    We call $\corrR$ an \define{$(r,s,t)$-correspondence} between $(X,D_X)$ and $(Y,D_Y)$ if, for all $(x,y)$, $(x',y') \in \corrR$,
    $|d_X(x,x')-d_Y(y,y')| \leq 2r$ and $\corrR$ is an $(s,t)$-correspondence between $D_X(x)$ and $D_X(y)$.

    We define a variant of the \define{Gromov-Hausdorff distance} between $(X,D_X)$ and $(Y,D_Y)$, denoted $\hat{d}_\mathrm{GH}((X,D_X),(Y,D_Y))$, to be the infimum over $r$ such that there is an $(r,r,r)$-correspondence between them.
\end{definition}

\subsection{Stability of Filtration Decorated Reeb Spaces}

To study the stability properties of the decorated Reeb space construction, we need to restrict ourselves to a certain subclass of $M$-metric fields.

\begin{definition}[$(L,\epsilon)$-Connectivity]
    We say $(X,f) \in \MF[M]$ has \define{$(L,\epsilon)$-connectivity} if for all $x,y \in X$, $\reebradius_f(x,y) \leq L d_X(x,y)+2\epsilon$. We denote the subspace of $\MF[M]$ consisting of $M$-metric fields with $(L,\epsilon)$-connectivity by $\MF[M]^{L,\epsilon}$.
\end{definition}

\begin{example}
    Suppose $(X,f)\in \MF[M]$ satisfies $\sup_{x \in X} d_M(f(x),g(x)) \leq \epsilon$ for some $L$-Lipschitz $g:X \to M$.
    If $X$ is a geodesic space, then $f$ has $(L,\epsilon)$-connectivity. 
    Indeed, let $\gamma$ be a geodesic in $X$ from $x$ to $y$. 
    Then $d_M(f(x),f(\gamma(t))) \leq d_M(g(x),g(\gamma(t)))+2\epsilon \leq Ld_X(x,y)+2\epsilon$. This implies that $\reebradius_f(x,y) \leq Ld_X(x,y)+2\epsilon$.
\end{example}


\begin{theorem}[Stability of Decorated Reeb Spaces]\label{thm:gh-stability}
    Let $(X,f),(Y,g) \in \MF[M]^{L,\epsilon}$. Let $\corrR$ be an $(r,s)$-correspondence between $(X,f)$ and $(Y,g)$ and let $\corrS$ be the induced correspondence between the Reeb spaces $\reebgraph_f$ and $\reebgraph_g$ given by $([x],[y]) \in \corrS$ whenever $(x,y) \in \corrR$. 
    Then, $\corrS$ is a $(2Lr+2s+2\epsilon,r,2Lr+2s+2\epsilon)$-correspondence between $(\reebgraph_f,D_f)$ and $(\reebgraph_g,D_g)$.
\end{theorem}

\begin{corollary}\label{cor:gh-stability}
    Let $(X,f),(Y,g) \in \MF[M]^{L,\epsilon}$. Then,
    \[ \hat{d}_\mathrm{GH}((\reebgraph_f,D_f),(\reebgraph_g,D_g)) \leq 2(L+1) \dgh((X,f),(Y,g)) + 2\epsilon. \]
\end{corollary}

\begin{proof}[Proof of Theorem \ref{thm:gh-stability}]
    Let $(x,y), (x',y') \in \corrR$. Let $\gamma:[0,1] \to Y$ be a path from $y$ to $y'$. For every $\delta > 0$ there is a partition 
    $0 = t_0 < t_1 < \dots < t_n =1$ 
    such that $\diam(g(\gamma([t_{i},t_{i+1}]))) \leq \delta$ for all $i$. 
    Let $y_i = \gamma(t_i)$ and $(x_i,y_i) \in \corrR$ such that $x_0=x$ and $x_n=x'$. Note that
    \[
    \reebradius_f(x_i,x_{i+1}) \leq L d_f(x_i,x_{i+1}) + \epsilon \leq L(2r+\delta)+2\epsilon, \qquad \forall i. 
    \]
    By definition of $\reebradius_f$, there exist paths $\alpha_i$ from $x_i$ to $x_{i+1}$ such that $\reebradius_f(x_i, \alpha_i(t)) \leq L(2r+\delta)+2\epsilon$ for all $t$. Hence, we have
    \begin{align*}
        d_M(f(x),f(\alpha_i(t))) &\leq d_M(f(x),f(x_i))+\reebradius_f(x_i, \alpha_i(t)) \\
        &\leq \sup_t d_M(g(x),g(\gamma(t))) + L(2r+\delta)+2s+2\epsilon.
    \end{align*}
    By concatenating $\alpha_i$'s, we get a path from $x$ to $x'$. Hence, the inequality above implies 
    \[\reebradius_f(x,x') \leq \sup_t d_M(g(x),g(\gamma(t))) + L(2r+\delta)+2s+2\epsilon. \]
    Since $\delta>0$ and $\gamma$ were arbitrary, we get $\reebradius_f(x,x') \leq \reebradius_g(y,y') + 2Lr+2s+2\epsilon$. One can similarly show that $\reebradius_g(y,y') \leq \reebradius_f(x,x') + 2Lr+2s+2\epsilon.$ Therefore $|\reebradius_f(x,x')-\reebradius_f(y,y')| \leq 2Lr + 2s +2\epsilon$.

    This implies that the metric distortion of $\corrS$ between $\reebgraph_f$ and $\reebgraph_g$ is less than or equal to $(4Lr+4s+4\epsilon)$, and $\corrR$ is an $(r, 2Lr+2s+2\epsilon)$-correspondence between $(X,D_f([x]))$ and $(Y,D_g([y]))$ for all $(x,y) \in \corrR$. Combining these two, we see that $\corrS$ is an $(2Lr+2s+2\epsilon,r,2Lr+2s+2\epsilon)$-correspondence between $(\reebgraph_f,D_f)$ and $(\reebgraph_g,D_g)$.
\end{proof}

\subsection{Stability of Barcode Decorated Reeb Spaces}

To establish the stability of barcode decorations (Definition \ref{def:barcode_decoration}), let us analyze the stability of the intermediate steps going from a metric filtration to a barcode.

\begin{definition}
    For any $\real$-metric field $f: X \to \real$, we have a two-parameter simplicial filtration $\{S_f^{t,u}\}_{t,u \geq 0}$ defined by $S_f^{t,u}:=\vr^t(\{x \in X \mid f(x) \leq u\})$.
\end{definition}

The one parameter version of this, where $t=u$, was shown to be stable in \cite{chazal2009gromov}. 
To establish a multi-parameter stability result, we utilize a multiparameter version of the homotopy interleaving distance (in the sense of \cite{frosini2019persistent} and \emph{not} in the sense of \cite{blumberg2023universality}) below.

\begin{definition}[cf.~\cite{frosini2019persistent}]
    Given two 2-parameter simplicial filtrations $S^{*,*}$ and $T^{*,*}$, an $(r,s)$-\define{homotopy interleaving} between them is a pair of 2-parameter families of simplicial maps $\phi^{t,u}: S^{t,u} \to T^{t+r,u+s}$ and  $\psi^{t,u}: T^{t,u} \to S^{t+r,s+u}$ that commute with the structure maps, and whose composition is homotopy equivalent to the inclusion maps $S^{*,*}\hookrightarrow S^{*+2r,*+2s}$  and $T^{*,*}\hookrightarrow T^{*+2r,*+2s}$. 
\end{definition}
We adapt the standard contiguity proof of Vietoris-Rips stability \cite{chazal2014persistence} to this multi-parameter setting.
\begin{proposition}\label{prop:simplicial-stability}
    Let $f: X \to \real$ and $g: Y \to \real$ be $\real$-metric fields and let $\corrR$ be an $(r,s)$-correspondence between them. Then $S_f^{*,*}$ and $S_g^{*,*}$ are $(2r,s)$-homotopy interleaved.
\end{proposition}
\begin{proof}
    Let $F: X \to Y$ and $G: Y \to X$ be functions such that $(x,F(x))$, $(G(y),y) \in \corrR$. Note that $F$ induces a simplicial map from $S_f^{t,u}$ to $S_g^{t+2r,u+s}$, and similarly $G$ induces a simplicial map. As in \cite{chazal2014persistence}, composition of these maps are contiguous to the inclusion, hence $S_f^{*,*}$ and $S_g^{*,*}$ are $(2r,s)$-homotopy interleaved. 
\end{proof}

As we already defined $(r,s)$-homotopy interleavings of simplicial filtrations, we can, by analogy, also define $(r,s)$-interleavings of persistence modules \cite{lesnick2015theory} by simply replacing each simplicial complex with its homology. 
Homotopy invariance of homology \cite{munkres2018elements} implies that if two simplicial filtrations are $(r,s)$-homotopy interleaved, then their persistent homology modules are $(r,s)$-interleaved. 
Assuming our modules are point-wise finite dimensional \cite{crawley2015decomposition}, we can induce barcodes by taking $1$-parameter slices of a $2$-parameter module, as in~\cite{landi2018rank}.

\begin{definition}
    Given a two parameter persistence module $P^{*,*}$ and $\lambda, c \geq 0$, we denote by $B_P^{\lambda, c}$  the persistence barcode of the $1$-parameter filtration $\{P^{t,\lambda t + c}\}_{t \geq 0}$. 
\end{definition}

The following result is an adaptation of \cite[Lemma~1]{landi2018rank} to our multiparameter setting. 

\begin{proposition}\label{prop:barcode-stability}
    Let $P^{*,*}$ and $Q^{*,*}$ be $(r,s)$-interleaved persistence modules. Given $\lambda,c \geq 0$,
    \[\db(B_P^{\lambda,c},B_Q^{\lambda,c}) \leq \max(r,s/\lambda), \]
    where $\db$ denotes the bottleneck distance.
\end{proposition}
\begin{proof}
    Let $\epsilon=\max(r,s/\lambda)$. It is enough to show that the $1$-parameter filtrations $V^*:=\{P^{t,\lambda t + c}\}_{t \geq 0}$ and $W^*:=\{Q^{t,\lambda t + c}\}_{t \geq 0}$ are $\epsilon$-interleaved. 
    Let $(\phi,\psi)$ be an $(r,s)$-interleaving between $P$ and $Q$. 
    Note that $t+r \leq t+\epsilon$ and $\lambda t + c + s \leq \lambda(t+\epsilon) + c$. 
    Hence, composing $\phi$ with the structure maps of $Q$, we get a map from $V^t \to W^{t+\epsilon}$ for all $t \geq 0$:
    \[\begin{tikzcd}
        V^t=P^{t,\lambda t + c} \arrow[r, "\phi"]  &  Q^{t+r, \lambda t + c + s} \arrow[r]& Q^{t+\epsilon,\lambda(t+\epsilon)+c} = W^{t+\epsilon}
    \end{tikzcd}\]
    Similarly we get maps from $W^*$ to $V^{*+\epsilon}$. Since $(\phi,\psi)$ is an interleaving, $V^*,W^*$ are $\epsilon$-interleaved.
\end{proof}

With the notation established above, for $[x] \in \reebgraph_f$, we have $B^{\lambda,c}_{f,k}([x]):= B^{\lambda,c}_{PH_k(S_{D_f([x])})}$.

\begin{theorem}\label{thm:barcode-decoration-stability}
For $(X,f),(Y,g) \in \MF[M]^{L,\epsilon}$, we have
\[\dgh((\reebgraph_f,B^{\lambda,c}_{f,k}),(\reebgraph_g,B^{\lambda,c}_{g,k})) \leq 2\max (1,1/\lambda) \big((L+1)\dgh((X,f),(Y,g))+\epsilon \big). \]
\end{theorem}
\begin{proof}
    Let $\corrR$ be an $(r,r)$-correspondence between $(X,f)$ and $(Y,g)$. By Theorem \ref{thm:gh-stability}, Proposition \ref{prop:simplicial-stability} and Proposition \ref{prop:barcode-stability}, $\corrR$ is an 
    $((2L+2)r+2\epsilon,\max(2r,\frac{(2L+2)r+2\epsilon}{\lambda}))$-correspondence between $(\reebgraph_f,B^{\lambda,c}_{f,k})$, $(\reebgraph_g,B^{\lambda,c}_{g,k})$. The maximum of $(2L+2)r+2\epsilon$ and $\frac{(2L+2)r+2\epsilon}{\lambda}$ is equal to $2\max(1,1/\lambda)((L+1)r+\epsilon)$. 
\end{proof}

\begin{remark}
    In the case that $M$ is a point, $\reebgraph_f,\reebgraph_g$ become singletons, $L,\epsilon$ can be taken to be zero, and $\lambda, c$ do not effect the barcode assigned to the single point, which becomes the barcode of the persistent homology of the Vietoris-Rips filtration of the corresponding metric space. So, taking $\lambda=1$, Theorem \ref{thm:barcode-decoration-stability} reduces to the classical result on the stability of Vietoris-Rips barcodes \cite{chazal2009gromov}.
\end{remark}

\section{Decorated Reeb Spaces for Combinatorial Graphs}

Having established the stability of the decorated Reeb space in the continuous setting, we now show how we can reliably approximate it in the finite setting. 
First, notice that the concept of Reeb radius makes sense for finite combinatorial graphs, where paths in a topological space are replaced by edge paths in a graph. 
More precisely, given a simple finite graph $G=(V,E)$ and a function $g: V \to M$ to a metric space $M$, we define $\reebradius_g: V \times V \to \real$ by
\[\reebradius_g(v,w):= \min \left\{\max_{i=1,\dots,n}d_M(g(v),g(v_i)) \mid v=v_0,\dots,v_n=w, \, [v_i,v_{i+1}] \in E \, \, \forall \, i \right\}. \] 
The Reeb space $\reebgraph_g$, metric $d_g$, and decorations $D_g$ and $B^{\lambda,c}_{g,k}$ are defined similarly.

\subsection{An Algorithm for the Reeb Radius Function}

Here we present an efficient algorithm for computing the Reeb radius function of a discrete graph by modifying Dijkstra's shortest path algorithm. 
See Figure \ref{fig:reeb_radius} for an example. A routine proof of the correctness and complexity of this algorithm is relegated to the appendix.

\begin{algorithm}
\caption{Input: Connected graph $G=(V,E)$, $g: V \to M$, $x \in V$.  Output: $\reebradius_g(x, \cdot): V \to \real$ }\label{alg:reeb_radius}
\begin{algorithmic}
\For {$v$ in $V$}:
    \State $\reebradius_g(x,v) \gets \infty$
\EndFor
\State $\reebradius_g(x,x) \gets 0$ 
\State $Q \gets [x]$

\While{$Q$ is non-empty}:
    \State $v \gets\underset{w \in Q}{\text{argmin }}\reebradius_g(x,w)$

    \State $Q \gets Q \setminus \{v\}$

    \For {$w$ in the set of neighbours of $v$}:
        \If {$\reebradius_g(x,w)=\infty$}:
          \State $\reebradius_g(x,w) \gets \max(\reebradius_g(x,v),d_M(g(x),g(w)))$
          \State $Q \gets Q\cup\{w\}$
        \EndIf
    \EndFor
\EndWhile

\State \textbf{return} $\reebradius_g(x,\cdot)$
\end{algorithmic}
\end{algorithm}

\begin{figure}[ht]
    \begin{center}
        \begin{tabular}{ccc}
            \includegraphics[width=0.3\linewidth]{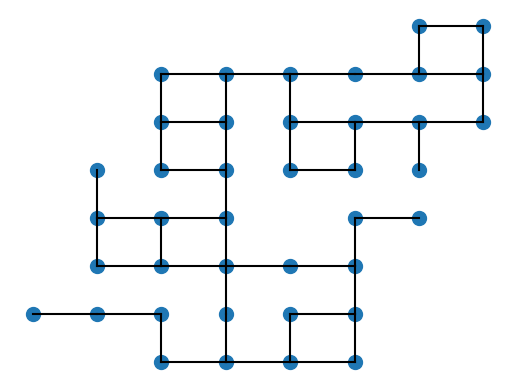} &
            \includegraphics[width=0.3\linewidth]{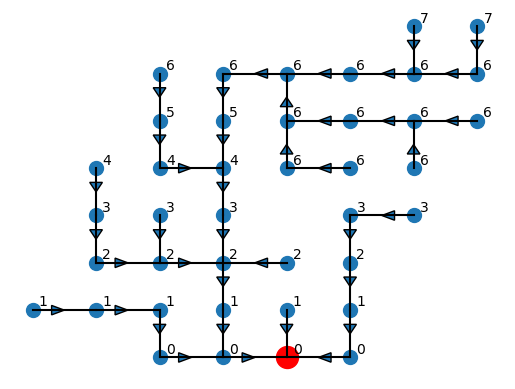} &
            \includegraphics[width=0.3\linewidth]{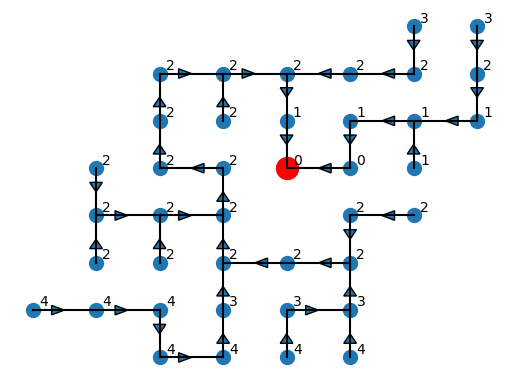}
        \end{tabular}
        \caption{On the left is a graph endowed with an $\real$-valued function $g$ given by node height. The other two graphs are the spanning trees generated by Algorithm \ref{alg:reeb_radius} with respect to the roote node $x$, with a directed edge $[w,v]$ for each $v,w$ that is processed inside the conditional statement. At each node, the value of $\reebradius_g(x,v)$ is written. The directed path from $v$ to $x$ realizes $\reebradius_g(x,v)$.}
        \label{fig:reeb_radius}
    \end{center}
\end{figure}

\subsection{Finite Approximations of Decorated Reeb Spaces}

The following result shows that, with a carefully constructed graph, we can approximate the Reeb radius function of a metric field.

\begin{proposition}\label{prop:graph-reeb-radius}
    Let $(X,f) \in \MF[M]^{L,\epsilon}$, let $\delta>0$, and let $V \subset X$ be a finite subspace such that for all $x \in X$, there exists $v \in V$ such that $d_X(x,v) < \delta$ and $d_M(f(x),f(v)) \leq \delta$. Let $G$ be the simple metric graph with node set $V$ and the edges given by $[v,w] \in E$ if $d_X(v,w) \leq 3\delta$. Let $g: V \to M$ be the restriction of $f$ to $V$. Then, for all $v,w \in V$, 
    \[|\reebradius_f(v,w)-\reebradius_g(v,w)| \leq (3L+1)\delta + \epsilon. \]
\end{proposition}
\begin{proof}
    Let $\gamma:[0,1] \to X$ be a path from $v$ to $w$. Take $0=t_0<\dots<t_n=1$ such that $d_X(\gamma(t_i),\gamma(t_{i+1})) \leq \delta$ for all $i$. Let $x_i = \gamma(t_i)$, and $v_i \in V$ such that $v_0=v$, $v_n=w$, $d_X(x_i,v_i) \leq \delta$, and $d_X(f(x_i),f(v_i)) \leq \delta$. Then $d_X(v_i,v_{i+1}) \leq 3\delta$, which means that $(v_0,\dots,v_n)$ is an edge path in $G$. We have
    \[\max_i d_M(g(v),g(v_i)) \leq \sup_t d_M(f(v),f(\gamma(t))) + \delta. \]
    Infimizing over $\gamma$, we get $\reebradius_g(v,w) \leq \reebradius_f(v,w)+\delta$.

    Let $(v=v_0,\dots,v_n=w)$ be an edge path in $G$. We have $\reebradius_f(v_i,v_{i+1}) \leq 3L\delta+\epsilon$. Hence, there is a path $\alpha_i$ from $x_i$ to $x_{i+1}$ in $X$ such that for all $x$ over the image of $\alpha_i$, we have
    \[d_M(f(v),f(x)) \leq d_M(f(v),f(x_i))+\reebradius_f(x_i,x) \leq d_M(f(v),f(x_i)) + 3L\delta+\epsilon.\]
    By concatenating $\alpha_i$'s, we get a path from $v$ to $w$ in $X$, hence
    \[\reebradius_f(v,w) \leq  \max_i d_M(f(v),f(x_i)) + 3L\delta+\epsilon.\]
    Infimizing over edge paths, we get $\reebradius_f(v,w) \leq \reebradius_g(v,w)+ 3L\delta+\epsilon.$
\end{proof}

\begin{theorem}\label{thm:finite-approximation}
    Let $(X,f)$ and $G=(V,E)$ with $g: V \to M$ be as in Proposition \ref{prop:graph-reeb-radius}. Then there is a $((5L+1)\delta+3\epsilon, \delta, (5L+1)\delta+3\epsilon)$-correspondence between $(\reebgraph_f,D_f)$ and $(\reebgraph_g,D_g)$.
\end{theorem}

\begin{proof}
    Let $\corrR$ be the correspondence between $X,V$ given by $(x,v) \in \corrR$ if $d_X(x,v) \leq \delta$. Let $(x,v)$, $(y,w) \in \corrR$. 
    \[\reebradius_f(x,y)-\reebradius_f(v,w) \leq \reebradius_f(x,v) + \reebradius_f(w,y) \leq Ld_X(x,v)+\epsilon+Ld_X(w,y)+\epsilon \leq 2(L\delta+\epsilon). \]
    By using the triangle inequality for $\reebradius_f$, one can see that $|\reebradius_f(v,w)-\reebradius_f(x,y)| \leq 2(L\delta+\epsilon)$. 
    Hence, by Proposition \ref{prop:graph-reeb-radius}, 
    \[|\reebradius_f(x,y)-\reebradius_g(v,w)| \leq (5L+1)\delta+3\epsilon. \]
    This shows that $\corrR$ is a $(\delta,(5L+1)\delta+3\epsilon)$-correspondence between $(X,\reebradius_f(x,\cdot))$ and $(V,\reebradius_g(v,\cdot))$. 
    Then, if $\corrS$ denotes the correspondence between $\reebgraph_f$ and $\reebgraph_g$ given by $([x],[y]) \in \corrS$ if $(x,y) \in \corrR$, then $\corrS$ is a $((5L+1)\delta+3\epsilon, \delta, (5L+1)\delta+3\epsilon)$-correspondence between $(\reebgraph_f,D_f)$ and $(\reebgraph_g,D_g)$.
\end{proof}

\begin{corollary}\label{cor:finite-barcode-approximation}
    Let $f: X \to M$ and $G=(V,E)$ with $g: V \to M$ be as in Proposition \ref{prop:graph-reeb-radius}. Then,
    \[\hat{d}_\mathrm{GH}((\reebgraph_f,B^{\lambda,c}_{f,k}),(\reebgraph_g,B^{\lambda,c}_{g,k})) \leq \max(1,1/\lambda) \big((5L+2)\delta+3\epsilon \big) \]
\end{corollary}
\begin{proof}
    Follows from Theorem \ref{thm:finite-approximation} and Proposition \ref{prop:simplicial-stability} and Proposition \ref{prop:barcode-stability}.
\end{proof}

\subsection{Simplifying Graphs by Smoothing}\label{sec:smoothing}

As we have seen above, the Reeb space construction can be adapted to the setting of (combinatorial) graphs by changing continuous paths to edge paths. 
As Reeb smoothing requires considering a certain subspace of $X \times M$, this approach does not directly generalize to graphs---as $M$ does not carry a graph structure, there is no natural way of seeing $G \times M$ as a finite graph. 
However, Theorem \ref{thm:reeb-radius-smoothings} gives a characterization of Reeb smoothings in terms of the Reeb radius, which does make sense in the graph setting.

\begin{definition}[Reeb smoothing of graphs]\label{def:reeb_smoothing_graphs}
    Let $G=(V,E)$ be a finite simple graph and $g: V \to M$, where $M$ is a metric space. We define $\reebgraph_g^\epsilon$ as the finite graph whose set of nodes $V^\epsilon_g$ is the quotient set $V/\sim$, where $v \sim w$ if $f(v)=f(w)$ and $\reebradius_g(v,w) \leq \epsilon$. 
    Denote the quotient map by $\pi_g^\epsilon: V \to V_g^\epsilon$. 
    The set of edges $E_g^\epsilon$ are induced by $E$ in the sense that if $[v,w] \in E$, then $[\pi_g^\epsilon(v),\pi_g^\epsilon(w)] \in E_g^\epsilon$. 
    We metrize $V_g^\epsilon$ by choosing a representative from each class and using $d_g$ distance between them. Let us denote such a metric by $d_g^\epsilon$. 
    If $V$ has a metric structure, we also use the barcode decorations of the selected representatives to get a barcode decoration $V_g^\epsilon$, which we denote by $D_{g^\epsilon,k}^{\lambda,c}$.
\end{definition}

\begin{remark}
    Note that the method of metrization and decoration described in Definition \ref{def:reeb_smoothing_graphs} depends on the representatives we choose, but it saves us from doing new computations, and choosing different representatives changes the metric at most by $\epsilon$.
\end{remark}

\begin{example}
    Figure \ref{fig:Reeb_Smoothings} shows a simple graph endowed with a function (height along the $y$-axis), its associated Reeb graph and two Reeb smoothings of the Reeb graph.
    \begin{figure}
        \centering
        \includegraphics[width = 0.9\textwidth]{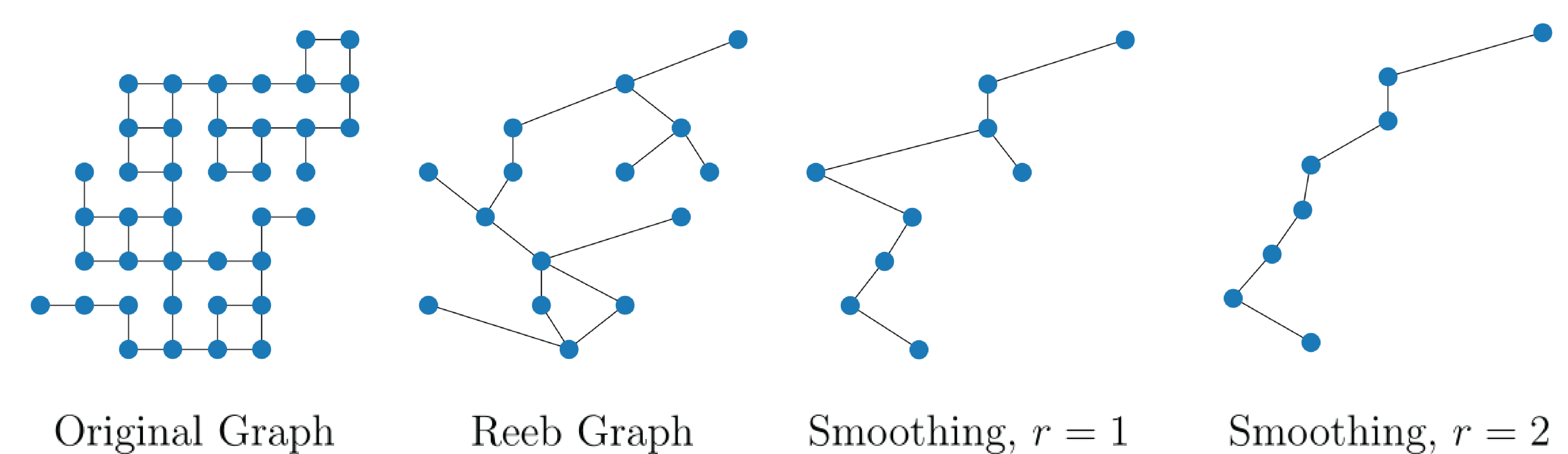}
        \caption{A graph with function $g$ given by node height, its Reeb graph and two smoothings.}
        \label{fig:Reeb_Smoothings}
    \end{figure}
\end{example}

The main advantage of using $\epsilon$-smoothings is that while the metric structure compared to the Reeb graph does not change much, the graph structure can be much simpler even for small $\epsilon$. 
For example, if the original function $g: V \to M$ on the node set takes distinct values for all vertices, then the Reeb graph will have the same graph structure as the original graph. 
However, it is possible that by changing values of $g$ by at most a small $\delta$, there could be many vertices taking the same value under $g$. The $\epsilon$-smoothing will then produce a much simpler graph. 
This is an alternative to the Mapper approach \cite{singh2007topological}, where a similar effect can be achieved by adjusting the bin size and connectivity parameters.

\begin{example}
    Figure \ref{fig:flight_data_smoothing} shows a Reeb smoothing of the \texttt{USAir97} graph~\cite{nr}. The nodes represent 332 US airports, with edges representing direct flights between them. We smooth the graph using PageRank~\cite{brin1998anatomy} as the function $g$ via the operation described in Definition \ref{def:reeb_smoothing_graphs}, and allowing small adjustments in $g$, as described above. The resulting smoothed graph clarifies the hub structure of the airport system. These results are similar to those of \cite{rosen2018homology}, where the graph was simplified using the Mapper algorithm.
\end{example}

   \begin{figure}
        \centering
        \includegraphics[width = 0.65\textwidth]{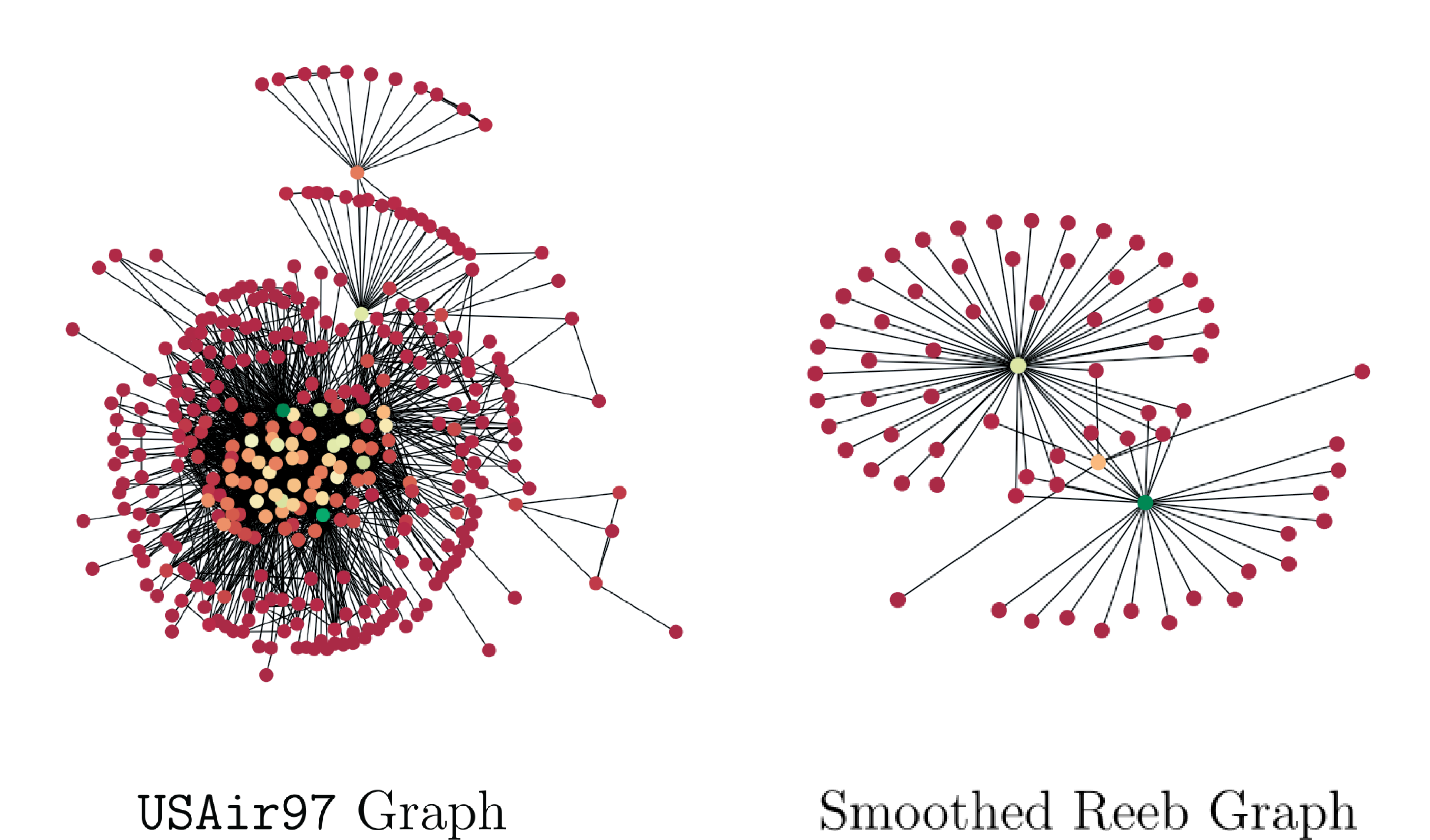}
        \caption{The \texttt{USAir97} graph of airports and connections, endowed 
        with the PageRank function, together with a smoothed Reeb graph illustrating hub structure.}
        \label{fig:flight_data_smoothing}
    \end{figure}

\begin{proposition}\label{prop:smoothing-approximation}
    Let $G=(V,E)$ be a finite simple graph and $g: V \to M$, where $V$ and $M$ are metric spaces. Let $h: V \to M$ such that $d_M(g(v),h(v)) \leq \delta$ for all $v \in V$. Then,
    \[\dgh((\reebgraph_g,B^{\lambda,c}_{g,k}),(\reebgraph_h^\epsilon,D_{h^\epsilon,k}^{\lambda,c}) \leq 2\max(1,1/\lambda)(\delta+\epsilon). \]
\end{proposition}
\begin{proof}
    Let $\corrR$ be the correspondence between $\reebgraph_g,\reebgraph_h^\epsilon$ consisting of elements of the form $(\pi_g(v),\pi_h^\epsilon(v))$ for $v \in V$. Let us show that $\corrR$ is an $(\delta+2\epsilon,0,(2\delta+\epsilon)/\lambda)$ correspondence between the corresponding decorated Reeb graphs. 
    
    Let $v,w \in \corrR$. Let $v',w' \in V$ such that $\pi_h^\epsilon(v')=\pi_h^\epsilon(v)$ and $\pi_h^\epsilon(w')=\pi_h^\epsilon(w)$. 
    Note that since $h(v)=h(v')$, we have $d_h([v],[v'])=2\reebradius_h(v,v') \leq 2\epsilon$. Similarly $d_h([w],[w']) \leq 2\epsilon$. Since $d_M(g(\cdot),h(\cdot)) \leq \delta$, we have $|d_g(v,w)-d_h(v,w)| \leq 2\delta$. Hence $|d_g([v],[w])-d_h([v'],[w'])| \leq 2\delta+4\epsilon$.
    For $x \in V$, we have
    \[|\reebradius_g(v,x)-\reebradius_h(v',x)| \leq |\reebradius_g(v,x)-\reebradius_h(v,x)| + |\reebradius_h(v,x)-\reebradius_h(v',x)| \leq 2\delta + \reebradius_h(v,v') \leq 2\delta+\epsilon. \]
    Hence, the identity correspondence is a $(0,2\delta+\epsilon)$ correspondence between $\reebradius_g(v,\cdot): V \to \real$ and $\reebradius_h(v',\cdot): V \to \real$. By Proposition \ref{prop:barcode-stability}, corresponding decorations have bottleneck distance at most $(2\delta+\epsilon)/\lambda$.
\end{proof}

\subsection{Computational Examples}\label{subsec:computation}

\begin{figure}
    \centering
    \includegraphics[width = 0.95\textwidth]{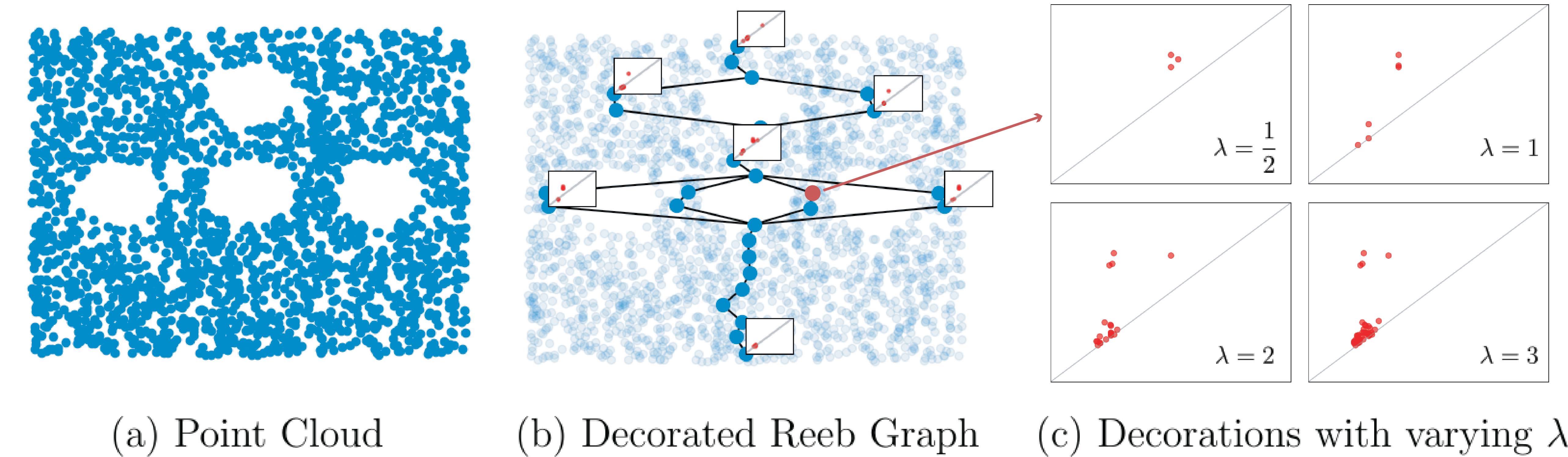}
    \caption{DRGs for synthetic data, illustrating dependence on the $\lambda$ parameter. (a) Synthetic 2D point cloud endowed with function given by height along the $y$-axis. (b) Smoothed Reeb graph, decorated with persistence diagrams $B_{f,1}^{1,0}$, using the notation of Definition \ref{def:barcode_decoration}. Barcodes are displayed for several nodes. (c) Barcode decorations of the form $B_{f,1}^{\lambda,0}$ for the indicated node, with varying $\lambda$.}
    \label{fig:Synthetic_DRG_Example}
\end{figure}

In this subsection, we present some simple examples which illustrate our computational pipeline. For the sake of simplicity, we focus on fields with $M=\real$, so that our examples are all Decorated Reeb Graphs.

\subparagraph*{Synthetic data and the effect of the $\lambda$-parameter.} Figure \ref{fig:Synthetic_DRG_Example} shows the DRG pipeline for a synthetic planar point cloud; the function $f$ is height along the $y$-axis. The Reeb graph in (b) is obtained by fitting a $k$-nearest neighbors graph to the point cloud ($k=10$, with respect to Euclidean distance), then smoothing the resulting (true) Reeb graph by rounding function values and applying the smoothing procedure described in Section \ref{sec:smoothing}. The decorations are degree-1 persistent homology diagrams $B_{f,1}^{\lambda,0}$ (Definition \ref{def:barcode_decoration}), for various values of $\lambda$. We see from this simple example that the decorations pick up the homology of the underlying point cloud from a localized perspective, where homological features that are closer to a node are accentuated in the diagram. Part (c) shows the effect of the $\lambda$ parameter on the decoration for a particular node; as $\lambda$ is increased, homological features appearing at larger Reeb radii from the node become more apparent.

\subparagraph*{Shape data and the effect of the function.} Figure \ref{fig:camel_examples} applies the DRG to a "camel" shape; this is a point cloud sampled from a shape in the classic mesh database from~\cite{sumner2004deformation}. 
The figure illustrates DRGs for two different functions on the shape: the first is height along the $z$-axis and the second is the $p$-eccentricity function, which for $p=2$ and a finite metric space $(X,d)$ is the function $x \mapsto \big(\sum_{y \in X} d(x,y)^2\big)^{1/2}$. 
This function gives a higher value to those points in the metric space that are far "on average" from other points in the space. 
In each case, the Reeb graph is approximated using a Mapper-like construction from \cite{curry2023topologically}, and decorations are of the form $B_{f,1}^{1,0}$, with respect to the appropriate function $f$. 
As expected, the resulting DRG depends heavily on the function, both in the shape of the Reeb graph and in the structure of the decorations.

\begin{figure}
    \centering
    \includegraphics[width = \textwidth]{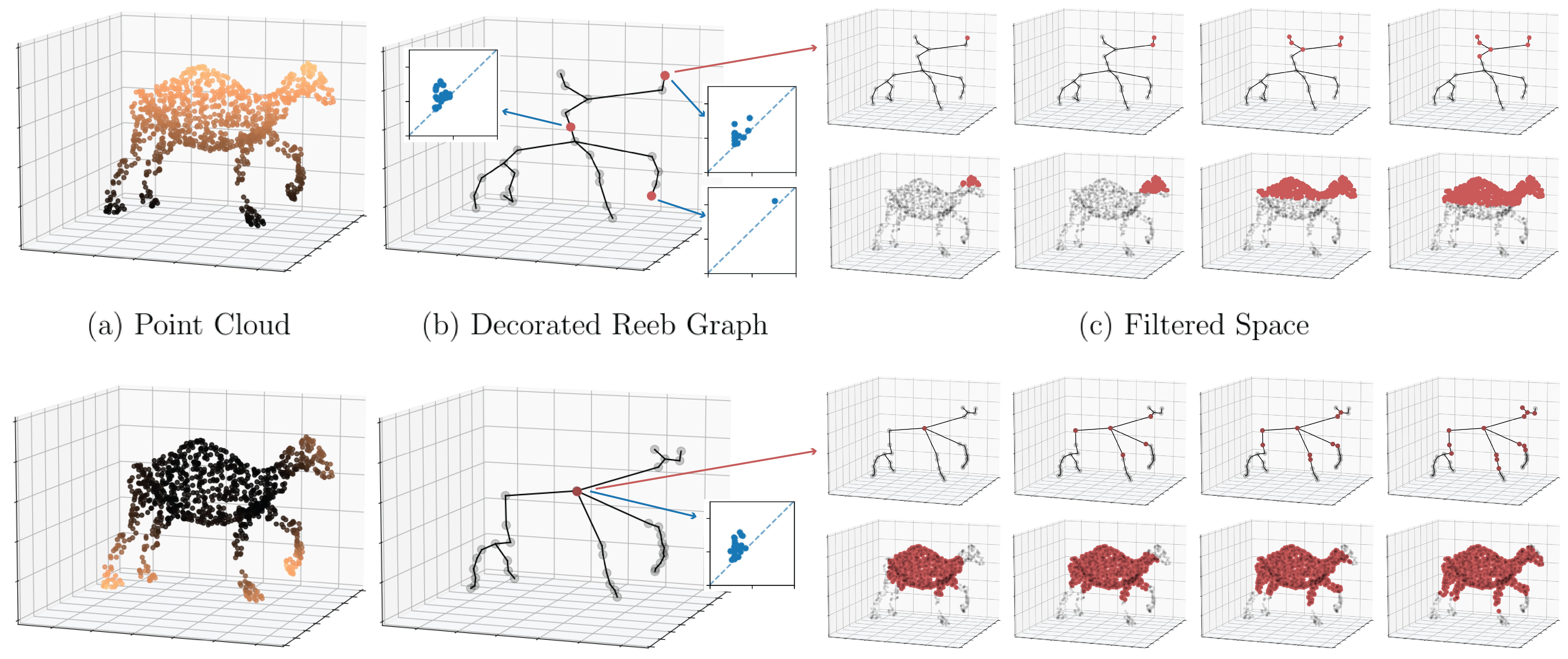}
    \caption{Illustration of DRGs and the effect of the function. {\bf Top Row:} (a) Point cloud sampled from a surface mesh of a camel shape from the Sumner database, endowed with function $f$ given by height along the $z$-axis. (b) The Reeb graph, decorated with persistence diagrams $B_{f,1}^{1,0}$, using the notation of Definition \ref{def:barcode_decoration}. Barcodes are displayed for the indicated nodes of the Reeb graph. (c) Shows the first steps in the filtered space associated to the indicated node. The barcode is computed by taking the Vietoris-Rips persistent homology with radius equal to the Reeb radius. {\bf Bottom Row:} Similar, but with function $g$ given by $p$-eccentricity with $p=2$. For the DRG in (b), all nodes except the indicated one have empty diagrams.}
    \label{fig:camel_examples}
\end{figure}

\subparagraph*{Distances between point clouds.} To illustrate the usage of the DRG pipeline in a metric-based analysis task, we perform a simple point cloud comparison experiment. The data consists of 10 samples each from 4 classes of meshes from the database of \cite{sumner2004deformation}. Each mesh is sampled to form a 3D point cloud, and each point cloud is endowed with a function $f$ given by height along the $z$-axis. 
A decorated Reeb graph is computed for each shape, first using the barcode decorations $B^{2,0}_{f,1}$. 
Each barcode decoration is then vectorized as a persistence image~\cite{adams2017persistence}, to improve computational efficiency, resulting in a Reeb graph with each node decorated by a vector in $\real^{25 \times 25}$. Actual computation of the field Gromov-Hausdorff distance (Definition \ref{def:field_GH}) is intractable, so we use Fused Gromov-Wasserstein (FGW) distance~\cite{vayer2020fused} as a more computationally convenient proxy. This distance is a variant of the Gromov-Wasserstein (GW) distance~\cite{memoli2011gromov}, which can be viewed, roughly, as an $L^p$-relaxation of Gromov-Hausdorff distance. The GW and FGW distances can be approximated efficiently via gradient descent and entropic regularization~\cite{flamary2021pot,peyre2016gromov}, and have been similarly used to compare topological signatures of datasets in several recent papers~\cite{curry2022decorated,curry2023topologically,li2023comparing,li2023flexible}.

Our results are shown in Figure \ref{fig:point_cloud_matching}. We show the pairwise distance matrix between DRGs with respect to FGW distance, as well as a multidimensional scaling plot. Observe that the classes are well-separated by their topological signatures. For comparison, we also compute GW distances between the underlying Reeb graphs; here, the classes are much less distinct, indicating that the decorations recover valuable distinguishing information about the shapes which is lost in the construction of their Reeb graphs. 

We note that this is only an illustrative example. To fully utilize the power of this framework, it is likely that feature optimization is necessary; for example, via graph neural networks, as in~\cite{curry2023topologically}. Further exploration of this technique will be the subject of future work.

\begin{figure}
    \centering
    \includegraphics[width = \textwidth]{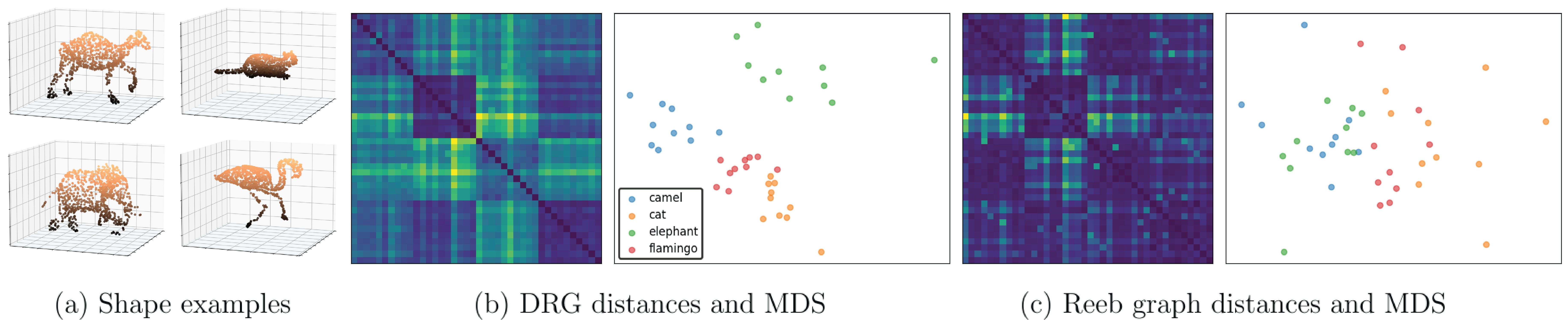}
    \caption{Point cloud matching results. (a) Example point clouds from each shape class. (b) The distance matrix and MDS embedding for decorated Reeb graph representations of the point clouds. (c) The distance matrix and MDS embedding using only Reeb graphs.}
    \label{fig:point_cloud_matching}
\end{figure}

\bibliography{biblio.bib}

 \appendix

 \section{Appendix}

 \begin{proposition}\label{prop:algorithm}
     Algorithm \ref{alg:reeb_radius} computes the correct function $\reebradius_g(x,\cdot):V \to \real$ with computational complexity $O(|E|+|V|\log|V|)$.
 \end{proposition}
 \begin{proof}
     Our proof is inspired by the proof of correctness and worst case complexity of Dijkstra's algorithm in \cite[Theorem 7.3 and 7.4]{korte2012combinatorial}.
    
     Assume we introduce a dictionary $p$ into the algorithm such that in the conditional statement, after inserting $w$ into $Q$,  we make the assignment $p(w)=v$. We call $p(w)$ the predecessor of $w$.

     If a node $w$ enters into $Q$,  it will leave $Q$ at one point and never enter into $Q$ again since $\reebradius_g(x,w)$ is set to a finite value. At the point where $w$ leaves $Q$, all of the neighbors of $w$ were either already inserted into $Q$ previously, or are inserted at the current iteration of the while loop. Since $G$ is connected, this implies that every node enters into and leaves $Q$ exactly once. Note that $p(w),p(p(w)),\cdots$ are removed from $Q$ before $w$ is inserted into $Q$, hence they are all distinct from $w$. This implies that $p(w),p(p(w)),\cdots$ does not have repeated elements, and terminates at $x$, as $x$ is the only node without a predecessor.

     For $v \in V$, let $R(v)$ denote the value of $\reebradius_g(x,v)$ set by the algorithm, and $\reebradius_g(x,v)$ denote the correct value. We need to show $R(v)=\reebradius_g(x,v)$. One can inductively see that $R(v) =\max_i\big(d_M(g(x),g(p^{(i)}(x)))\big)$, where $p^{(i)}$ denotes applying $p$ $i$-times. Hence $\reebradius_g(x,v) \leq R(v)$. To show the inverse inequality, we use the following claim.
    
     \textbf{Claim} If $[v,w] \in E$, then $R(p(w)) \leq R(v)$.

     First, note that $p(w)$ is the first neighbor of $w$ removed from $Q$, since $w$ enters into $Q$ exactly at that point. At the iteration $p(w)$ is removed from $Q$, $R$ takes the minimal value at $p(w)$ in $Q$, and inductively one can see that any value that $R$ takes on $Q$ at the same or later iterations is greater than or equal to $R(p(w))$. Either $v=p(w)$ or it is removed later from $Q$. In both cases, we have $R(p(w)) \leq R(v)$, proving the claim.

     Now, take the edge path $x=v_0,\dots,v_n=v$ realizing $\reebradius_g(x,v)$. Using the claim we get
     \begin{equation*}
         \begin{split}
             R(v) &= \max\big(R(p(v)), d_M(g(x),g(v)) \big) \leq \max \big(R(v_{n-1}), d_M(g(x),g(v_n)) \big) \\
                  &\leq \max(R(v_{n-2}), d_M(g(x),g(v_{n-1}),d_M(g(x),g(v_n))) \leq \dots \\
                  &\leq \max_i \big( d_M(g(x),g(v_i) \big) = \reebradius_g(x,v).
         \end{split}
     \end{equation*}
     This shows the correctness of the algorithm.

     For the complexity analysis, note that there are exactly $|V|$ iterations of the while loop. Assuming the function $\reebradius_g(x,\cdot)$ over $Q$ is implemented using a Fibonacci heap, at the iteration of the while loop where $v$ is removed, $O(log(|V|))$ steps are needed for getting and removing the argmin, and $O(deg(v))$ operations are needed in the for loop, as every operation inside the for loop is $O(1)$. Adding all these, we get the total complexity of $O(|E|+|V|log(|V|))$.
 \end{proof}

\end{document}